\documentclass[11pt, a4paper]{article}
\usepackage{amsfonts}
\usepackage{}
\usepackage{amsthm}
\usepackage{amsmath,amssymb,latexsym,color}
\usepackage[mathscr]{eucal}
\usepackage[colorlinks,
            linkcolor=blue,
            anchorcolor=green,
            citecolor=magenta
           ]{hyperref}
\usepackage{graphicx}
\usepackage{tikz}
\usetikzlibrary{calc}
\usepackage{authblk}

\long\def\delete#1{}

\usepackage{color}

\definecolor{Blue}{rgb}{0,0,1}
\definecolor{Red}{rgb}{1,0,0}
\definecolor{DarkGreen}{rgb}{0,0.6,0}
\definecolor{DarkYellow}{rgb}{1,1,0.2}
\definecolor{DarkPurple}{rgb}{.6,0,1}

\usepackage{xcolor}
\usepackage[normalem]{ulem}

\usepackage{cleveref}
\crefformat{section}{\S#2#1#3}
\crefformat{subsection}{\S#2#1#3}
\crefformat{subsubsection}{\S#2#1#3}
\crefrangeformat{section}{\S\S#3#1#4 to~#5#2#6}
\crefmultiformat{section}{\S\S#2#1#3}{ and~#2#1#3}{, #2#1#3}{ and~#2#1#3}

\def\q{\hfill\rule{1ex}{1ex}}

\begin{document}
\setcounter{page}{1}
\newtheorem{thm}{Theorem}[section]
\newtheorem{fthm}[thm]{Fundamental Theorem}
\newtheorem{dfn}[thm]{Definition}
\newtheorem*{rem}{Remark}
\newtheorem{lem}[thm]{Lemma}
\newtheorem{cor}[thm]{Corollary}
\newtheorem{exa}[thm]{Example}
\newtheorem{prop}[thm]{Proposition}
\newtheorem{prob}[thm]{Problem}
\newtheorem{fact}[section]{Fact}
\newtheorem{con}[thm]{Conjecture}
\renewcommand{\thefootnote}{}
\newcommand{\remark}{\vspace{2ex}\noindent{\bf Remark.\quad}}
\newtheorem{ob}[thm]{Observation}
\newcommand{\rmnum}[1]{\romannumeral #1}
\renewcommand{\abovewithdelims}[2]{%
\genfrac{[}{]}{0pt}{}{#1}{#2}}

\newcommand\Sy{\mathrm{S}}
\newcommand\Cay{\mathrm{Cay}}
\newcommand\tw{\mathrm{tw}}
\newcommand\supp{\mathrm{supp}}


\def\qed{\hfill$\Box$\vspace{11pt}}

\title {\bf  Treewidth of the $q$-Kneser graphs}

\author{Mengyu Cao\thanks{ E-mail: \texttt{caomengyu@mail.bnu.edu.cn}}}
\author{Ke Liu\thanks{E-mail: \texttt{liuke17@mails.tsinghua.edu.cn}}}
\author{Mei Lu\thanks{E-mail: \texttt{lumei@tsinghua.edu.cn}}}
\author{Zequn Lv\thanks{Corresponding author. E-mail: \texttt{lvzq19@mails.tsinghua.edu.cn}}}

\affil{\small Department of Mathematical Sciences, Tsinghua University, Beijing 100084, China}

\date{}

\openup 0.5\jot
\maketitle

\begin{abstract}
Let $V$ be an $n$-dimensional vector space over a finite field $\mathbb{F}_q$, where $q$ is  a prime power. Define  the \emph{generalized $q$-Kneser graph} $K_q(n,k,t)$ to be the graph whose vertices are the $k$-dimensional subspaces of $V$ and two vertices $F_1$ and $F_2$ are adjacent if $\dim(F_1\cap F_2)<t$. Then $K_q(n,k,1)$ is the well-known $q$-Kneser graph. In this paper,
 we determine the  treewidth of $K_q(n,k,t)$ for $n\geq 2t(k-t+1)+k+1$ and $t\ge 1$ exactly. Note that  $K_q(n,k,k-1)$ is  the complement of the Grassmann graph $G_q(n,k)$. We give a more precise result for the treewidth of $\overline{G_q(n,k)}$ for any possible $n$, $k$ and $q$.

\vspace{2mm}


\


\end{abstract}

\section{Introduction}

Let $G=(V,E)$ be a finite, simple and undirected graph. For $v\in V$, the degree of $v$ in $G$, written as $d_G(v)$, is the number of edges incident with $v$ in $G$.  Let $\Delta(G)$ be the maximum degree of $G$.   A subset $S$ of $V(G)$ is called \emph{independent set} if no two elements in $S$ are adjacent in $G$.  The \emph{independence number} of $G$, denoted by $\alpha(G)$, is the  maximum size of independent sets in $G$. We will use  $\overline{G}$ to denote the complement of  $G$.

\begin{dfn}
{\em Let $G$ be a graph, $T$ a tree and $(V_{t})_{t\in V(T)}$ be a family of vertex sets  $V_{t}\subseteq V(G)$ indexed by the vertices $t\in V(T)$. A \emph{tree decomposition} of  $G$ is a pair $(T,(V_{t})_{t\in V(T)})$ if it  satisfies the following three conditions:
\begin{itemize}
\item[{\rm(i)}] $V(G)=\cup_{t\in V(T)}V_{t}$;
\item[{\rm(ii)}] for every edge $uv\in E(G)$, there is a $t\in V(T)$ such that $u,v\in V_{t}$;
\item[{\rm(iii)}] for every $v\in V(G)$, the subgraph of $T$ induced by $\{t\in V(T)\mid v\in V_{t}\}$ is  connected.
\end{itemize}}
\end{dfn}

The \emph{width} of the decomposition $(T,(B_{t})_{t\in V(T)})$ is  $\max\{|B_{t}|\mid t\in V(T)\}-1.$ The \emph{treewidth} of  $G$, denoted by $\tw(G)$, is the least width of any tree decompositions of $G$. 

Treewidth is a well-studied parameter in modern graph theory that shows the alikeness of the graph with a tree.  On the one hand, it is an important variable in structural graph theory. Treewidth was introduced by Robertson and Seymour
in their series of fundamental papers on graph minors, for example, we refer the reader to  \cite{Robertson3,Robertson2,Robertson}.  On the other hand, treewidth is also a key parameter in  algorithm design. The problem of deciding whether a graph has tree decomposition of treewidth at most $k$ is NP-complete \cite{A}. Besides, it has been shown that many NP-hard combinatorial problems can be solved in polynomial time for graph with treewidth bounded by a constant \cite{Bodlaender,Bo}. In the past few decades, there are lots of literatures investigate the treewidth of certain graphs, for example, \cite{Wood,Wood2,Kloks,Li,Mitsche,Wood3}. However, it is difficult to determine the treewidth exactly in most situations, and there are only few papers obtained the exact  treewidth of some certain graphs. In 2014, the  treewidth of the Kneser graphs were determined exactly by Harvey et al. \cite{Wood}. In 2020, Liu et al. \cite{Liu} determined the exact treewidth of the generalized Kneser graphs. Motivated by these two results, we study the exact value of treewidth of the generalized $q$-Kneser graphs in this paper.

 Let $n,k\in \mathbb{Z}^{\rm +}$ with $1 \le k \le n$, and $V$ an $n$-dimensional vector space over the finite field $\mathbb{F}_q$, where $q$ is  a prime power. Denote by ${V\brack k}$  the family of all $k$-dimensional subspaces of $V$.  In the sequel we will abbreviate ``$k$-dimensional subspace'' to ``$k$-subspace''. Let  $a,b\in \mathbb{Z}^{\rm +}$. The \emph{Gaussian binomial coefficient} is defined as
$$
{a\brack b} = \prod_{0\leq i<b}\frac{q^{a-i}-1}{q^{b-i}-1}.
$$
In addition, we set ${a\brack 0}=1$ and ${a\brack c} =0$ if $c<0$. Recall that $|{V\brack k}|={n\brack k}$. For any $t\in \mathbb{Z}^{\rm +}$, a family $\mathcal{F}\subseteq{V\brack k}$ is said to be $t$-\emph{intersecting} if $\dim(A\cap B)\geq t$ for all $A,B\in\mathcal{F}$.

Let $n,k,t\in \mathbb{Z}$ with $1\leq t< k \leq n$. Write $[n]=\{1,2,\ldots,n\}$ and denote by ${[n]\choose k}$ the collection of all $k$-subsets of $[n].$ The \emph{generalized Kneser graph}, denoted by $K(n,k,t)$, is a graph whose vertex set is ${[n]\choose k}$ and two vertices $A$ and $B$ are adjacent if $|A\cap B|<t$. The graph $K(n,k,1)$ is the well-known \emph{Kneser graph} \cite{Kneser,Lovasz}. Define  $K_q(n,k,t)$ to be the \emph{generalized $q$-Kneser graph} for $1\leq t<k$ whose vertex set is ${V\brack k}$ and two vertices $F_1$ and $F_2$ are adjacent if $\dim(F_1\cap F_2)<t$. When $t=1$,  $K_q(n,k,1)$ is the well-known \emph{$q$-Kneser graph}. When $t=k-1$, the graph $K_q(n,k,k-1)$, usually denoted by $\overline{G_q(n,k)}$, is the complement of the  \emph{Grassmann graph} $G_q(n,k)$.
  Over the years several aspects of $q$-Kneser graphs and Grassmann graphs such as chromatic number, energy, eigenvalues and some other properties had been widely studied as one can find in, for example, \cite{Blokhuis,Huang,Lv,Lv2,Numata,Tanaka}.

We  know that the  famous  Erd\H{o}s-Ko-Rado  Theorem \cite{Erdos-Ko-Rado-1961-313} has a well-known relationship to the independent number of the generalized Kneser graphs, since an independent set in the generalized Kneser graph $K(n,k,t)$ is a $t$-intersecting family of ${[n]\choose k}$. Similarly, the  Erd\H{o}s-Ko-Rado  Theorem for vector spaces \cite{Tanaka-2006-903} also has a well-known relationship to the independent number of the generalized $q$-Kneser graphs. In this paper, we will use such relationship to obtained the following two main results.

\begin{thm}\label{GKneser}
Let $k>t\geq 1$ and $n\geq 2t(k-t+1)+k+1$. Then $$\tw(K_q(n, k,t))= {n\brack k}-{n-t\brack k-t}-1.$$
\end{thm}

Another result is about the exact value of  $\tw(\overline{G_q(n,k)})$ for any $n$ and $k$. Note that $\overline{G_q(n,k)}$ is an null graph when $n<k+2$. Thus we only consider the case with $n\geq k+2$.

\begin{thm}\label{CJohnson}
Let $n$ and $k$ be positive integers with $n\geq k+2$ and $k\geq 2$. Then
$$
{\tw}(\overline{G_q(n,k)}) = \begin{cases}
q^4+O(q^3), & \mbox{if}\ k=2 \ \mbox{and}\ n=4,\\
{n\brack k}-{n-k+1\brack 1}-1, & \mbox{otherwise}.
\end{cases}
$$
\end{thm}

The rest of this paper is organized as follows. In Section \ref{thm1}, we will give the exact value of $\tw(K_q(n, k,t))$ for $k>t\geq1$ and large $n$ corresponding to $k$ and $t$. After that, we will study the treewidth of the complement of Grassman graphs for any possible $n$ and $k$ in Section \ref{thm2}.

\section{Treewidth of $K_q(n, k,t)$}\label{thm1}

\subsection{Upper bound for treewidth in Theorem~\ref{GKneser}}
In this subsection, we  give an upper bound  of $\tw(K_q(n, k,t))$ with the help of the following famous Erd\H{o}s-Ko-Rado Theorem for vector spaces.

\begin{thm}{\rm (\cite{Tanaka-2006-903})}\label{EKR-vectorspace}
Let $n,k,t\in \mathbb{Z}^{\rm +}$ with $2k\leq n$ and $t\leq k.$ If $\mathcal{F}\subseteq{V\brack k}$ is $t$-intersecting, then
$$
|\mathcal{F}|\leq {n-t\brack k-t}.
$$
Equality holds if and only if either
\begin{itemize}
\item[{\rm(i)}] $\mathcal{F}$ consists of all $k$-subspaces of $V$ which contain a fixed $t$-subspace of $V$, or
\item[{\rm(ii)}] $n = 2k$ and $\mathcal{F}$ consists of all $k$-subspaces of a
fixed $(n-t)$-subspace of $V$.
\end{itemize}
\end{thm}
The result of  Theorem \ref{EKR-vectorspace} is clearly equivalent to the independent number of the generalized $q$-Kneser graph $K_q(n,k,t)$. That is,
\begin{equation}\label{independent}
\alpha(K_q(n,k,t))={n-t\brack k-t}
\end{equation}
for $n\geq 2k$.

The following lemma can be easily proved.

\begin{lem}\label{lem1-1-1}
Let $m$ and $i$ be positive integers with $ m\ge i.$ Then the following results hold:
\begin{itemize}
\item[{\rm (i)}] ${m\brack i}={m-1\brack i-1}+q^i{m-1\brack i}$ and ${m\brack i}=\frac{q^m-1}{q^i-1}\cdot{m-1\brack i-1}$;
\item[{\rm (ii)}] $q^{m-i}<\frac{q^m-1}{q^i-1}<q^{m-i+1}$ and $q^{i-m-1}<\frac{q^i-1}{q^m-1}<q^{i-m}$ if $i <m$;
\item[{\rm(iii)}] $q^{i(m-i)}\leq{m\brack i}< q^{i(m-i+1)}$, and $q^{i(m-i)}<{m\brack i}$ if $i <m$.
\end{itemize}
\end{lem}

\begin{prop}{\rm(\cite[Lemma~9.3.2]{Brouwer})}\label{lem4}
Suppose $0\leq i,j\leq n.$ If $X$ is a $j$-subspace of $V$, then there are precisely $q^{(i-m)(j-m)}{n-j\brack i-m}{j\brack m}$ $i$-subspaces $Y$ in $V$ such that $\dim(X\cap Y)=m$.
\end{prop}

\begin{prop}{\rm (\cite{Wood})}\label{upperbound}
For any graph $G$, $\tw(G)\leq\max\{\Delta(G),|V(G)|-\alpha(G)-1\}$.
\end{prop}

\begin{lem}\label{upbound}
If $n,k$ and $t$ be positive integers with $2k\leq n$ and $t\leq k$, then
\begin{equation}\label{upperbound1}
\tw(K_q(n,k,t))\leq {n\brack k}-{n-t\brack k-t}-1.
\end{equation}
\end{lem}
\begin{proof}
According to Proposition~\ref{upperbound}, to prove an upper bound of $\tw(K_q(n,k,t))$ we only need to compare the size of $\Delta(K_q(n,k,t))$ and $|V(K_q(n,k,t))|-\alpha(K_q(n,k,t))-1.$

Fix $A\in {V\brack k}$, by Proposition \ref{lem4}, we have $$\left|\left\{B\in{V\brack k}\mid \dim(A\cap B)=i\right\}\right|=q^{(k-i)^2}{n-k\brack k-i}{k\brack i}.$$ By the definition of $K_q(n,k,t)$, we have
\begin{align*}
\Delta(K_q(n,k,t))&=\sum\limits_{i=0}^{t-1}\left|\left\{B\in{V\brack k}\mid \dim(A\cap B)=i\right\}\right|\\
&=\sum\limits_{i=0}^{t-1}q^{(k-i)^2}{n-k\brack k-i}{k\brack i}.
\end{align*}

\noindent\textbf{Claim 1.}  For any $t\geq 1$, we have $q^{(k-t)^2}{n-k\brack k-t}{k\brack t}> {n-t\brack k-t}$.

\vskip.2cm
{\bf Proof of Claim 1.} On the one hand, by Lemma \ref{lem1-1-1} (iii), we have
$$q^{(k-t)^2}{n-k\brack k-t}{k\brack t}> q^{(k-t)^2}q^{(k-t)(n-2k+t)}q^{t(k-t)}=q^{(k-t)(n-k+t)}.$$
On the other hand, by Lemma \ref{lem1-1-1} (iii) again, we get
$${n-t\brack k-t}< q^{(k-t)(n-k+1)}.$$
Thus, we have the result as required. \q

It suffices to prove that $\Delta(K_q(n,k,t))+\alpha(K_q(n,k,t))< |V(K_q(n,k,t))|$. According to the Claim 1, we have
\begin{align*}
\Delta(K_q(n,k,t))+\alpha(K_q(n,k,t))&< \sum\limits_{i=0}^{t-1}q^{(k-i)^2}{n-k\brack k-i}{k\brack i}+q^{(k-t)^2}{n-k\brack k-t}{k\brack t}\\
&=\sum\limits_{i=0}^{t}q^{(k-i)^2}{n-k\brack k-i}{k\brack i}\\
&=\sum\limits_{i=0}^{t}\left|\left\{B\in{V\brack k}\mid \dim(A\cap B)=i\right\}\right|.
\end{align*}
Since $\bigcup\limits_{i=0}^{t}\left\{B\in{V\brack k}\mid \dim(A\cap B)=i\right\}\subseteq {V\brack k},$  we have $$\Delta(K_q(n,k,t))+\alpha(K_q(n,k,t))< |V(K_q(n,k,t))|,$$
and $\tw(K_q(n,k,t))\leq |V(K_q(n,k,t))|-\alpha(K_q(n,k,t))-1={n\brack k}-{n-t\brack k-t}-1.$
\end{proof}
\subsection{Lower bound for treewidth in Theorem~\ref{GKneser}}\label{LB1}
In this subsection, we give the lower bound  of $\tw(K_q(n,k,t))$.

Let $X\subseteq V(G)$ and $G[X]$ the subgraph of $G$ induced by $X$. Denots $G-X=G[V(G)\setminus X]$.  Let $p$ be a fixed constant with $\frac{2}{3}\leq p < 1$. The \emph{$p$-separator} of  $G$ is  a subset $X\subset V(G)$ such that there is no component in $G-X$ that contains more than $p|V(G-X)|$ vertices. The following result describes the relationship  between the treewidth and the $p$-separators of  $G$.

\begin{prop}{\rm (\cite{Robertson})}\label{separator}
Every graph $G$ has a $p$-separator of order at most $\tw(G)+1$ for each $\frac{2}{3}\leq p < 1$.
\end{prop}

Let $X$ be a $p$-separator of $G$. Then  the vertices of $G-X$ can be divided into two parts, say $\mathcal {A}$ and $\mathcal {B}$, such that the components in $\mathcal {A}$ and $\mathcal {B}$ contain at most $p|V(G-X)|$ vertices. This leads to the following result.
\begin{lem}\label{separator2}
Let $X$ be a $p$-separator. Then the vertices of $G-X$ can be divided into two parts $\mathcal {A}$ and $\mathcal {B}$ such that there is no edge between $\mathcal {A}$ and $\mathcal {B}$, and the equation
\begin{align}
\frac{1}{3}|V(G-X)|\leq &|\mathcal {A}|,|\mathcal {B}|\leq \frac{2}{3}|V(G-X)| \label{e2'}
\end{align}
holds.
\end{lem}

With the help of these important results we give the following lemma.

\begin{lem}\label{lowerbound}
Let $k>t\geq 1$ and $n\geq 2t(k-t+1)+k+1$. Let $\Gamma:=K_q(n, k,t)$. Then $$\tw(\Gamma)\geq {n\brack k}-{n-t\brack k-t}-1.$$
\end{lem}
\begin{proof}
We suppose to the contrary that $\tw(\Gamma)< {n\brack k}-{n-t\brack k-t}-1.$ By Proposition \ref{separator}, there is a $\frac{2}{3}$-separator $X$ such that $|X|<{n\brack k}-{n-t\brack k-t}.$ Therefore, $|V(\Gamma-X)|>{n-t\brack k-t}=\alpha(\Gamma)$, and we have $V(\Gamma-X)$ is too large to be an independent set which implies that there exists an edge $u_1u_2$ in $\Gamma-X$. By Lemma~\ref{separator2}, $V(\Gamma-X)$ can be divided into two parts $\mathcal {A}$ and $\mathcal {B}$ such that there is no edge between $\mathcal {A}$ and $\mathcal {B}$, and the equation (\ref{e2'}) holds. Thus, $|\mathcal {A}|,|\mathcal {B}|\geq 2$. Without loss of generality, assume that  $u_1u_2$ is in $G[\mathcal {A}]$, where $G=\Gamma-X$.

Let $S={u_1 \brack t}\times {u_2 \brack t}$. Thus, $|S|={k\brack t}^2$. For any $v\in \mathcal {B}$ and $i\in\{1,2\}$, since $vu_i\notin E(\Gamma)$, we have $\dim(v\cap u_i)\geq t$. As $\dim(u_1\cap u_2)<t$, we get $\tau_1\neq \tau_2$ for any $(\tau_1,\tau_2)\in S$, which implies that $\dim(\tau_1\cap\tau_2)\leq t-1$. Let $\mathcal {B}(\tau_1,\tau_2)=\{v\in \mathcal {B} \mid \tau_1,\tau_2\in v\}$. According to Pigeonhole Principle, there exists some $(\tau_1,\tau_2)\in S$ such that
$$|\mathcal {B}(\tau_1,\tau_2)|\geq \frac{1}{|S|}|\mathcal {B}|>{k\brack t}^{-2}\cdot\frac{1}{3} {n-t\brack k-t},$$
by (\ref{e2'}). On the other hand, since $\dim(\tau_1+\tau_2)$ is minimized when $\dim(\tau_1\cap\tau_2)$ is maximized and $$\dim(\tau_1+\tau_2)=\dim(\tau_1)+\dim(\tau_2)-\dim(\tau_1\cap\tau_2)\geq t+1,$$
we have $|\mathcal {B}(\tau_1,\tau_2)|\leq {n-t-1\brack k-t-1}$. Combining the lower  and upper bounds of $|\mathcal {B}(\tau_1,\tau_2)|$, we have
\begin{equation}\label{e1}
{n-t-1\brack k-t-1}> {k\brack t}^{-2}\cdot\frac{1}{3} {n-t\brack k-t}.
\end{equation}

\noindent\textbf{Claim 2.}  If $k>t\geq 1$ and $n\geq 2t(k-t+1)+k+1$, then ${n-t-1\brack k-t-1}\leq {k\brack t}^{-2}\cdot\frac{1}{3} {n-t\brack k-t}.$

\vskip.2cm
{\bf Proof of Claim 2.} It suffices to prove that $\frac{q^{n-t}-1}{q^{k-t}-1}\geq 3{k\brack t}^2$ by Lemma \ref{lem1-1-1} (i).

According to Lemma \ref{lem1-1-1} (ii), $\frac{q^{n-t}-1}{q^{k-t}-1}> q^{n-k}$.  By Lemma \ref{lem1-1-1} (iii), we have $3{k\brack t}^2<3\cdot q^{2t(k-t+1)}$. If $q\geq 3$, we have the result as required since $n\geq 2t(k-t+1)+k+1$. If $q=2$, then
\begin{align*}
3{k\brack t}^2&=3\cdot \prod\limits^{t-1}_{i=0}\frac{2^{k-i}-1}{2^{t-i}-1}\cdot\prod\limits^{t-1}_{i=0}\frac{2^{k-i}-1}{2^{t-i}-1}\\
&=\prod\limits^{t-1}_{i=0}\frac{2^{k-i}-1}{2^{t-i}-1}\cdot\prod\limits^{t-3}_{i=0}\frac{2^{k-i}-1}{2^{t-i}-1}2^{k-t+1}2^{k-t+2}\\
&\leq 2^{t(k-t+1)}2^{(t-2)(k-t+1)}2^{k-t+1}2^{k-t+2}\\
&=2^{2t(k-t+1)+1}.
\end{align*}
Therefore, we also have the result as required. \q

By Claim 2, we have a contradiction with the equation (\ref{e1}) .
\end{proof}

\noindent\textit{Proof of Theorem~\ref{GKneser}.}\quad
By Lemmas \ref{upbound} and \ref{lowerbound}, we obtain the result directly. \qed

\section{Treewidth of the complement of Grassmann graphs}\label{thm2}
In this section, we study the treewidth of the complement of Grassmann graphs, and give the exact value of the treewidth of $\overline{G_q(n,k)}$ for $n\geq k+2$. Note that $\overline{G_q(n,k)}$ is an empty graph when $n<k+2$. Firstly, we can easily get the upper bound of $\tw(\overline{G_q(n,k)})$ by Lemma 2.5 since $\overline{G_q(n,k)}=K_q(n,k,k-1)$.

\begin{lem}\label{up2}
Let $n$ and $k$ be positive integers with $k\geq2$ and $n\geq 2k$. Then
$$\tw(\overline{G_q(n,k)})\leq{n\brack k}-{n-k+1\brack 1}-1.$$
\end{lem}

\begin{lem}\label{lb2}
Let $n$ and $k$ be positive integers with $k\geq2$ and $n\geq \max\{k+3,2k\}$. Then
$$\tw(\overline{G_q(n,k)})\geq{n\brack k}-{n-k+1\brack 1}-1.$$
\end{lem}
\begin{proof}
Suppose to the contrary that $\tw(\overline{G_q(n,k)})<{n\brack k}-{n-k+1\brack 1}-1.$ By Proposition \ref{separator}, there exists a $\frac{2}{3}$-separator $X$ such that $|X|<{n\brack k}-{n-k+1\brack 1}$. Therefore, $|V(\overline{G_q(n,k)}-X)|>{n-k+1\brack 1}$. Let $G=\overline{G_q(n,k)}-X$ for short. By Lemma~\ref{separator2}, it is easy to see that $V(G)$ can be partitioned into two parts $\mathcal {A}$ and $\mathcal {B}$ such that there is no edge between $\mathcal {A}$ and $\mathcal {B}$, and the equations
\begin{align}
\frac{1}{3}|V(G)|\leq &|\mathcal {A}|,|\mathcal {B}|\leq \frac{2}{3}|V(G)| \label{e2}
\end{align}
holds. Thus, $|\mathcal {A}|,|\mathcal {B}|\geq 2$. By Theorem~\ref{EKR-vectorspace} and $n\geq \max\{k+3,2k\}\geq 2k$, we have  $\alpha(\overline{G_q(n,k)})={n-k+1\brack 1}$. Since $|V(G)|>{n-k+1\brack 1}$,
 there is an edge in  $G[\mathcal {A}]$ or  $G[\mathcal {B}]$. Without loss of generality, assume that  $v_1v_2$ is in $G[\mathcal {A}]$ and let $v_1\cap v_2=w$. Thus we have $\dim w\leq k-2$.

We claim that $\dim w= k-2$. Since for any vertex $u\in \mathcal {B}$, there is no vertex in $\mathcal {A}$ is adjacent to $u$ in $G$, we have $\dim(u\cap v_1),\dim(u\cap v_2)\geq k-1$. This implies that $\dim(u\cap v_1),\dim(u\cap v_2)= k-1$ and then
\begin{align}\label{e3}
\dim(v_1\cap v_2)&\geq \dim((u\cap v_1)\cap(u\cap v_2))\\\notag
&\geq \dim(u\cap v_1)+\dim(u\cap v_2)-\dim u\\\notag
&=k-2.
\end{align}
On the other hand, $\dim w\leq k-2$ from above. Thus we have $\dim w=k-2$, as required.

Let $\alpha_1,\alpha_2,\hdots,\alpha_{k-2}$ be a basis of $w=v_1\cap v_2,$ then $w=\langle\alpha_1,\alpha_2,\hdots,\alpha_{k-2} \rangle$ and we let $$v_1=w+\langle  \beta_1, \beta_2 \rangle~\mbox{ and }~v_2=w+\langle  \beta_3, \beta_4 \rangle.$$
For any $u\in \mathcal {B},$ by equation (\ref{e3}), we have $\dim(u\cap v_1\cap v_2)\geq k-2$. On the other hand, since $\dim (v_1\cap v_2)= k-2$, we have $\dim(u\cap v_1\cap v_2)\leq k-2$. And then $\dim(u\cap v_1\cap v_2)= k-2$, which implies that $$|\mathcal {B}|\leq {2\brack 1}\times {2\brack 1}=(q+1)^2.$$

 We will complete  the proof by considering  the following two cases.

\vskip.2cm
\noindent\textbf{Case 1.} $G[\mathcal {B}]$ contains an edge.
\vskip.2cm

Let $u_1u_2$ be an edge in $G[\mathcal {B}]$. Since $u_iv_j\notin E(G)$ for $1\le i,j\le 2$, we can assume, without loss of generality, that
$$\begin{array}{rcl}
u_1=w+\langle \beta_1,\beta_3 \rangle,&u_2=w+\langle \beta_2,\beta_4 \rangle.
\end{array}$$
Thus we get $w+<\beta_1,\beta_2,\beta_3,\beta_4>=u_1+u_2=v_1+v_2.$ Following the similar analysis above, we have $|\mathcal {A}|\leq(q+1)^2$.
 Define
\begin{align*}
\mathcal{A}(u_1,u_2)&=\left\{u\in {u_1+u_2\brack k}\mid w\subseteq u,~\dim(u\cap u_1)=k-1,~\dim(u\cap u_2)=k-1\right\},\\
\mathcal{B}(v_1,v_2)&=\left\{v\in {v_1+v_2\brack k}\mid w\subseteq v,~\dim(v\cap v_1)=k-1,~\dim(v\cap v_2)=k-1\right\}.
\end{align*}
We can
 easily see that $\mathcal{A}\subseteq\mathcal{A}(u_1,u_2)$, $\mathcal{B}\subseteq\mathcal{B}(v_1,v_2)$ and $|\mathcal{A}(u_1,u_2)|\le (1+q)^2$, $|\mathcal{B}(u_1,u_2)|\le (1+q)^2$.
 Denote $z_1=w+\langle \beta_1,\beta_4 \rangle$ and $z_2=w+\langle \beta_2,\beta_3 \rangle$. Then $z_1,z_2\in \mathcal{A}(u_1,u_2)\cap \mathcal{B}(v_1,v_2)$.
Let $z_3=w+\langle \beta_1+\beta_2,\beta_3+\beta_4 \rangle$ and $z_4=w+\langle \beta_2,\beta_1+\beta_3 \rangle$. Then $z_3\in \mathcal{A}(u_1,u_2)\setminus \mathcal{B}(v_1,v_2)$ and $z_4\in \mathcal{B}(v_1,v_2)\setminus \mathcal{A}(u_1,u_2)$. Since $\dim(z_3\cap z_4)=k-2,$ we have $z_3$ is adjacent to $z_4$. Therefore, if $z_3\in \mathcal{A},$ then $z_4\notin \mathcal{B}$ which implies $z_4\notin \mathcal{A}(u_1,u_2)\cup \mathcal{B}(v_1,v_2)$. Thus we have
$$|V(G)|\leq |\mathcal {A}\cup\mathcal {B}|\le |\mathcal{A}(u_1,u_2)|+| \mathcal{B}(v_1,v_2)|-|\{z_1,z_2,z\}|\le 2(1+q)^2-3,$$ where $z\in \{z_3,z_4\}$.

On the other hand, by the assumption in the beginning, $|V(G)|\geq{n-k+1\brack 1}+1$. Combining with the upper  and  lower bounds of $|V(G)|$, we have
\begin{align*}
2(1+q)^2-3\geq {n-k+1\brack 1}+1=q^{n-k}+q^{n-k-1}+\cdots+q+2,
\end{align*}
a contradiction with $n\geq k+3.$

\vskip.2cm
\noindent\textbf{Case 2.} $\mathcal {B}$ is an independent set in $G$.
\vskip.2cm

Let $u_1,u_2\in \mathcal {B}$ and $\alpha=u_1\cap u_2$. Then $\dim\alpha=k-1$ and $\dim(u_1+ u_2)=k+1$. We define
\begin{align*}
\mathcal {C}&=\left\{u\in {V\brack k}\mid \alpha \in u\right\}\setminus {u_1+u_2\brack k},\\
\mathcal {D}&={u_1+u_2\brack k}.
\end{align*}
And let $\mathcal {D}_1=\{u\in \mathcal {D} \mid \alpha \in u\}$ and $\mathcal {D}_2=\mathcal {D}\setminus \mathcal {D}_1.$ It is easy to see that $\mathcal {C}$ and  $\mathcal {D}$ are independent sets in $G$, $|\mathcal {D}|={k+1\brack 1}$ and $|\mathcal {D}_1|={2\brack 1}=q+1$.

Firstly we prove that $\mathcal {A}\cup\mathcal {B}\subseteq\mathcal {C}\cup\mathcal {D}$. Clearly, $u_1,u_2\in \mathcal{C}\cup\mathcal {D}$. For any $s\in \mathcal {A}\cup\mathcal {B}$ and $s\notin \{u_1,u_2\}$, we prove that $s\in \mathcal {C}\cup\mathcal {D}.$ If $\alpha \subseteq s$, then we have $s\in \mathcal {C}\cup\mathcal {D}$. If $\alpha \nsubseteq s$, then $\dim (s\cap u_1\cap u_2)=\dim(s\cap \alpha)\leq k-2$, and we have
\begin{align*}
\dim((s\cap u_1)+(s\cap u_2))&=\dim(s\cap u_1)+\dim(s\cap u_2)-\dim(s\cap u_1\cap u_2)\\
&\geq 2(k-1)-(k-2)\\
&=k.
\end{align*}
Furthermore, since $(s\cap u_1)+(s\cap u_2)\subseteq s$ and $\dim s=k$,  we have $s=(s\cap u_1)+(s\cap u_2)\in {u_1+u_2\brack k}\subseteq \mathcal {C}\cup\mathcal {D}$.

Next, for any $y\in \mathcal {D}_2$, we have $\dim (y\cap \alpha)\leq k-2$. On the other hand, $\dim (y\cap \alpha)=\dim y +\dim \alpha-\dim(y+\alpha)\geq k+(k-1)-(k+1)=k-2$. So we have $\dim (y\cap \alpha)=k-2$.  Hence for  any $x\in \mathcal {C}$,  $\dim(x\cap y)\leq k-2$, which implies that $x$ is adjacent to $y$.
 Furthermore, since for any $x\in \mathcal {C}$ and $z\in \mathcal {D}_1$, we can easily see that $x$ is not adjacent to $z$. As there is an edge in $G[\mathcal {A}]$, we have $\mathcal {A}\cap \mathcal {C}\neq \emptyset$ and $\mathcal {A}\cap \mathcal{D}_2\neq \emptyset$. Therefore, if there is $x\in \mathcal {B}\cap \mathcal {C}$, then there is a vertex $y\in \mathcal {A}\cap \mathcal{D}_2$ such that $xy\in E(G)$; if there is a vertex $x\in \mathcal {B}\cap \mathcal{D}_2$, then there is a vertex $y\in \mathcal {A}\cap \mathcal {C}$ such that $xy\in E(G)$. Thus, we have $\mathcal {B}\subseteq \mathcal {D}_1$, since there is no edge between $\mathcal {A}$ and $\mathcal {B}$. Therefore, we get $|\mathcal {B}|\leq |\mathcal {D}_1|=q+1$. On the other hand,
$$|\mathcal {B}|\geq \frac{1}{3}|G|\geq\frac{1}{3} \left({n-k+1\brack 1}+1\right).$$

Combining with the lower  and  upper bounds of $|\mathcal {B}|,$ we have a contradiction with our assumption that $n\geq k+3$.
\end{proof}

\noindent\textit{Proof of Theorem~\ref{CJohnson}.}\quad We divide the proof of this theorem into the following two cases.

\vskip.2cm
\noindent\textbf{Case 1.} $k\geq 3.$
\vskip.2cm

In this case we have $2k\geq k+3$. If $n\geq 2k$, by Lemmas \ref{up2} and \ref{lb2}, we have $\tw(\overline{G_q(n,k)})={n\brack k}-{n-k+1\brack 1}-1.$ Note that $\overline{G_q(n,k)}\cong\overline{G_q(n,n-k)}$. If $n<2k$ then  $n>2(n-k)$.  By  Lemmas \ref{up2} and \ref{lb2},  $\tw(\overline{G_q(n,k)})=\tw(\overline{G_q(n,n-k)})={n\brack k}-{n-k+1\brack 1}-1.$

\vskip.2cm
\noindent\textbf{Case 2.} $k=2.$
\vskip.2cm

If $n\geq 5=k+3$, by Lemmas \ref{up2} and \ref{lb2}, we get $\tw(\overline{G_q(n,2)})={n\brack 2}-{n-1\brack 1}-1.$ If $n=4$, then $\tw(\overline{G_q(4,2)})\leq {4\brack 2}-{3\brack 1}-1=q^4+q^3+q^2-1$. There is a trivial lower bound of $\tw(\overline{G_q(4,2)})$, that is $\tw(\overline{G_q(4,2)})\geq \delta (\tw(\overline{G_q(4,2)}))=q^4$. Furthermore, following the similar proof of Lemma~\ref{lb2}, we have  $\tw(\overline{G_q(4,2)})\geq {4\brack 2}-{4\brack 1}-1=q^4+q^2-1$. Therefore, we have $\tw(\overline{G_q(4,2)})=q^4+O(q^3).$

Consequently, we complete the proof of this theorem. \qed

\section*{Acknowledgement}
This research was supported by   the National Natural Science Foundation of China (Grant 11771247 \& 11971158) and  Tsinghua University Initiative Scientific Research Program.

\end{document}